\documentclass[11pt]{amsart}
\usepackage{amsmath}
\usepackage{amsfonts}
\usepackage{amssymb}
\usepackage{amsthm}
\usepackage{color}
\usepackage[parfill]{parskip}
\usepackage{etoolbox}
\usepackage{scrextend}

\usepackage{xcolor}
\usepackage{color}

\usepackage[dvips]{graphics}
\usepackage[dvips]{graphicx}

\usepackage[utf8]{inputenc} 
\usepackage[T1]{fontenc}
\usepackage[all]{xy}

\newtheorem{theorem}{Theorem}[section]
\newtheorem{corollary}[theorem]{Corollary}
\newtheorem{lemma}[theorem]{Lemma}
\newtheorem{prop}[theorem]{Proposition}

\newtheorem{definition}[theorem]{Definition}

\newtheorem{remark}[theorem]{Remark}

\newtheorem*{theorem*}{Theorem}

\newcommand{\op}{\overline{\Bbb P}_n}
\newcommand{\up}{\underline{\Bbb P}_n}

\newcommand{\St}{{\rm {St}}}
\newcommand{\lk}{{\rm {Lk}}}

\newcommand{\Z}{{\Bbb Z}}

\newcommand{\E}{{\Bbb E}}
\newcommand{\U}{{\mathcal U}}
\newcommand{\R}{\Bbb R}

\newcommand{\N}{{\mathcal N}}
\newcommand{\cc}{{\frak c}}

\begin{document}
\title{Random Simplicial Complexes in the Medial Regime}
\author{Michael Farber} 
\address{School of Mathematical Sciences \\
Queen Mary, University of London\\
London, E1 4NS\\
United Kingdom}
\email{m.farber@qmul.ac.uk}
\thanks{Michael Farber was partially supported by a grant from the Leverhulme Foundation.}
\author{Lewis Mead}
\address{School of Mathematical Sciences \\
Queen Mary, University of London\\
London, E1 4NS\\
United Kingdom}
\email{lewis.mead@qmul.ac.uk}
\maketitle
\begin{abstract}
%A recent paper \cite{farber} analyses two distinct ways for generating random simplicial complexes out of a Bernoulli random hypergraph $X$: 
%one may take the largest simplicial subcomplex $\underline X$ contained in $X$ and alternatively, one may consider 
%the smallest simplicial complex $\overline X$ containing $X$. This leads to two distinct probability distributions on the set of simplicial complexes, the lower and upper probability measures. It was shown in \cite{farber} that the lower and upper random simplicial complexes are Alexander dual to each other. The Betti numbers of the upper and lower random simplicial complexes were also analysed under certain assumptions. 
We describe topology of random simplicial complexes in the lower and upper models in the medial regime, i.e. under the assumption that the probability parameters $p_\sigma$ approach neither $0$ nor $1$. 
We show that nontrivial Betti numbers of typical lower and upper random simplicial complexes in the medial regime lie in a narrow range of dimensions. For instance, an upper random simplicial complex $Y$ on $n$ vertices in the medial regime with high probability has non-vanishing Betti numbers $b_{j}(Y)$ only for 
$k+c <n-j<k+\log_2 k +c'$ where $k=\log_2 \ln n$ and $c, c' $ are constants. 
A lower random simplicial complex on $n$ vertices in the medial regime is with high probability 
$(k+a)$-connected and its dimension $d$ satisfies $d\sim k+\log_2 k+ a'$ where $a, \, a'$ are constants. 
The paper develops a new technique, based on Alexander duality, which relates the lower and upper models. 

%The technique used in proof for the upper model is based 
%
%The proofs for the upper model use the new technique based on Alexander duality and the results for the lower model. 
\end{abstract}
\section{Introduction}
Several models of high-dimensional random simplicial complexes have been intensely studied in recent years by many authors. 
This study is chiefly motivated by the need to model large complex systems in various scientific and industrial applications.
It is currently well understood that random simplicial complexes provide a more flexible mathematical modelling tool compared to random graphs, which are widely used. Methods of random topology may also be useful
in pure mathematics where they enable construction of objects with rare combination topological properties.

Historically the first models of random simplicial complexes were suggested by Linial and Meshulam \cite{linial} and Meshulam and Wallach \cite{meshulam}.  More recently Costa and Farber \cite{costa, costa1,costa2,costa3} studied a multi-parameter generalisation of the Linial-Meshulam-Wallach models involving a sequence of probability parameters $p_0, p_1, p_2, \dots$, each of the parameters controlling the density of simplexes of the corresponding dimension. The multi-parameter random simplical complex includes also the random clique complex \cite{kahle} as a special case. 

A multi-parameter random simplicial complex $Y$ can be constructed as follows. 
One starts with an $n$-dimensional simplex $\Delta_n$ where $n\to \infty$. Consider the set of all $i$-dimensional faces of $\Delta_n$ and choose its random subset selecting each $i$-dimensional 
sub-simplex of $\Delta_n$ with probability $p_i$ independently of each other. Making such selection for every dimension $i\le n$ we obtain a random hypergraph $X$ and the multiparameter random simplicial complex $Y$ is defined as the largest simplicial complex contained in $X$. In other words, a simplex $\sigma\subset \Delta_n$ belongs to $Y$ iff its every face $\tau\subseteq \sigma$ belongs to $X$. 

The process described in the previous paragraph is homogeneous version of {\it \lq\lq the lower model\rq\rq} of random simplicial complexes, in terminology of the recent paper \cite{farber}. There is also {\it \lq\lq an upper model \rq\rq}, see \cite{farber}, where the random simplicial complex $Y'$ is obtained from the random hypergraph $X$ as above by taking the smallest simplicial complex $Y'\subset \Delta_n$ containing $X$ (contrary to the largest simplicial subcomplex $Y$ contained in $X$). 
A simplex $\sigma\subset \Delta_n$ belongs to $Y'$ iff 
there is a larger simplex $\tau\supseteq \sigma$ which belongs to $X$. 
It was shown in \cite{farber} that the lower and upper models are Alexander dual to each other and thus knowing the Betti number in one of the models immediately gives an answer for the other model. 

In this paper we do not aim to survey all current activity and progress in the research field of topology of random simplicial complexes. It is a vibrant and rapidly developing area with many publications; we refer the reader to a recent survey \cite{kahle1}. While the majority of publications  study various specifications of the lower model, very recently results concerning special cases of the upper model have started to emerge, see \cite{cooley}.

The notion of {\it a 
 critical dimension} characterises the global behaviour of the Betti numbers in the lower and upper models, under certain assumptions, see
 \cite{costa3} and \cite{farber} .
 In the lower model, the Betti numbers below the critical dimension vanish and Betti numbers above the critical dimension are significantly smaller 
than the Betti number in the critical dimension. The structure of the Betti numbers in the upper model is slightly more complicated (see \cite{farber} where the notion of {\it a spread} is introduced); it can be vaguely characterised by saying that homologically a lower and upper random simplicial complexes are well approximated by wedges of spheres of the critical dimension. Note that unlike the present paper, the main assumption of \cite{costa3}, \cite{farber} was that the probability parameters $p_i$ have the form $p_i=n^{-\alpha_i}$ and in particular they tend to $0$ as $n\to \infty$. 
 
In the present paper we study the opposite situation: we assume that the probability parameters satisfy 
\begin{equation}\label{eqnmedial}
p\leq p_\sigma \leq P
\end{equation}
for all simplexes $\sigma$ 
 where the numbers
 $p, \, \, P\in(0,1)$ are independent of $n$. 
 We call this \textit{the medial regime}. In the medial regime the probability parameters $p_\sigma$ can approach neither $0$ nor $1$.
 Note that the medial regime includes as a special case the simplest and most natural situation when all probability parameters $0<p_\sigma<1$ are equal to each other and are independent of $n$. 
 
 We show that {\it a lower model} random simplicial complex $Y$ in the medial regime has dimension 
 $$\dim Y\, \sim\,  \log_2\ln n+\log_2\log_2\ln n,$$ it
 is simply connected and, may have nontrivial Betti numbers 
 $b_j(Y)$ only for 
 \begin{eqnarray}\label{between}
 j\in \left[\log_2\ln n +c, \, \, \log_2\ln n+\log_2\log_2\ln n+c'\right],
 \end{eqnarray}
 where $c, c'$ are constants. 
A more precise statement is given below as Theorem \ref{thm1}. The proof uses the Garland method relating the spectral gap of links with vanishing of the Betti numbers. 

We also describe topology of a typical random simplicial complex $Y$ with respect to {\it the upper model} in the medial regime. We show that it has a rather different behaviour: (a) its dimension equals $n-1$, (b) it contains the skeleton $\Delta_n^{(n-d)}$ where %$d$ is approximately 
$$d \sim \log_2\ln n+\log_2\log_2\ln n,$$
(c) the Betti numbers $b_{n-j}(Y)$ vanish except for a range of dimensions of width approximately $\log_2\log_2\ln n$. 
A precise statement is given below as Theorem \ref{thm2}. 

The proofs of the main results of this paper concerning the lower model use Garland's method which has been used recently 
by other authors working in stochastic topology. Additionally, we employ a new tool, Alexander duality, which allows us to deduce the results concerning the upper model; as far as we know this duality technique is new within the area of probabilistic topology.

\section{Random simplicial complexes: the upper  and lower models.}\label{sec:twomodels}

Here we recall the construction of the lower and upper probability measures on the set of simplicial complexes following \cite{farber}. 
\subsection{} Let $[n]$ denote  the set $ \{0, 1,\dots,n\}$. 
{\it A hypergraph} $X=\{\sigma\}_{\sigma\subsetneq [n]}$ is a collection 
of non-empty proper subsets  $\sigma\subsetneq [n]$. We emphasise that for technical reasons we exclude the possibility for a hypergraph $X$ to contain the 
whole set $[n]$. 
The symbol
$\Omega_n$ stands for the set of all such hypergraphs. 
%Next we define a probability measure on $\Omega_n$.  
Let $p_\sigma\in [0,1]$ be a probability parameter associated with each non-empty proper subset $\sigma\subsetneq [n]$. 
Using these parameters %$p_\sigma$ 
one defines a probability function $\mathbb P_n$ on $\Omega_n$ by the formula
\begin{eqnarray}\label{prob0} \mathbb{P}_n(X) =\prod_{\sigma\in X}p_\sigma\cdot\prod_{\sigma\not\in X} q_\sigma, \quad \mbox{where}\quad q_\sigma=1-p_\sigma.\end{eqnarray}
%Here $q_\sigma$ denotes $ 1- p_\sigma$. 
%Formula (\ref{prob0}) can be described by saying that each simplex $\sigma\subseteq [n]$ is included into a random hypergraph $X$ 
%with probability $p_\sigma$ 
%independently of all other simplexes. 
Clearly, $\mathbb{P}_n$ is a Bernouilli measure on the set of all non-empty subsets of $[n]$.

\subsection{} Let $\Omega^\ast_n\subseteq \Omega_n$ denote the set of all {\it simplicial complexes} on the vertex set $[n]=\{0, 1, \dots, n\}$. Recall that a hypergraph $X$ is a simplicial complex if it is closed with respect to taking faces.%, i.e. if $\sigma\in X$ and $\tau\subseteq \sigma$ imply that  $\tau\in X$. 

Let $\Delta_n$ denote the simplicial complex consisting of all non-empty subsets of $[n]$. The complex $\Delta_n$ is known as the $n$-dimensional simplex spanned by the set $[n]$. The set $\Omega_n^\ast$ is the set of all subcomplexes of $\Delta_n$. 
Non-empty subsets of $[n]$ will be also referred to as simplexes. 

There are two natural retractions 
\begin{eqnarray}\label{maps}\overline r, \, \, \underline r: \, \Omega_n\to \Omega^\ast_n\end{eqnarray}
which are defined as follows.
For a hypergraph $X\in \Omega_n$ we denote by 
$\overline r(X) =\overline X$ the smallest simplicial complex %in $\Omega^\ast_n$ 
containing $X$; simplex $\sigma\in \Delta_n$ belongs to $\overline X$ iff for some $\tau\in X$ one has $\tau\supseteq \sigma$. 
On the other hand, the simplicial complex $\underline r(X) =\underline X$ is defined as the largest simplicial complex 
%in $\Omega^\ast_n$ 
contained in $X$; a simplex $\sigma\subseteq [n]$ belongs to $\underline X$ iff every simplex $\tau\subseteq \sigma$ belongs to $X$. 

%One has 
%\begin{eqnarray}
%\underline X \subseteq X \subseteq \overline X.
%\end{eqnarray}

We shall denote by 
\begin{eqnarray}\label{measures}
\overline {\Bbb P}_n= \overline r_\ast({\Bbb P}_n)\quad \mbox{and}\quad \underline {\Bbb P}_n= \underline r_\ast({\Bbb P}_n)
\end{eqnarray}
the two probability measures on the space of simplicial complexes $\Omega^\ast_n$ obtained as the push-forwards (or image measures) of the measure (\ref{prob0})
with respect to the maps (\ref{maps}). Explicitly, for a simplicial complex $Y\subseteq \Delta_n$ one has
\begin{eqnarray}\label{explicitly}
\op(Y) =\sum_{X\in \Omega_n, \, \overline X=Y} \mathbb P_n(X) \quad \mbox{and}\quad \up(Y) =\sum_{X\in \Omega_n, \,\underline X=Y} \mathbb P_n(X).\end{eqnarray}
We call $\op$ and $\up$ {\it the upper and lower measures} correspondingly. %the upper probability

There are explicit formulae for the lower and upper probability measures. For a simplicial complex $Y\subseteq \Delta_n$ one has (see \cite{farber}):
\begin{equation}\label{form1}
\up(Y) = \prod_{\sigma \in Y} p_\sigma \cdot \prod_{\sigma\in E(Y)} q_\sigma, \quad \mbox{and}\quad
\op(Y) = \prod_{\sigma \in M(Y)} p_\sigma \cdot \prod_{\sigma\not\in Y} q_\sigma.
\end{equation}
Here $E(Y)$ denotes the set of {\it external simplexes} of a simplicial subcomplex $Y\subseteq \Delta_n$, i.e. simplexes $\sigma\in \Delta_n$ such that $\sigma\not\in Y$ but the boundary $\partial \sigma$ is contained in $Y$. 
The symbol $M(Y)$ denotes the set of {\it maximal simplexes} of $Y$, i.e. those which are not faces of other simplexes of $Y$. 

Random simplicial complexes $Y\in \Omega_n^\ast$ with respect of the lower probability measures $\up$ admit the following intuitive description. The complex $Y$ is a union of its skeleta 
$$Y^0\subset Y^1\subset Y^2\subset \dots\subset Y^{n-1}=Y$$ where $Y^0$ is a random $0$-dimensional 
complex obtained from the set of vertices $[n]=\{0, 1, \dots, n\}$ by selecting each vertex $v\in [n]$ at random, with probability $p_v$, independently of the other vertices. For $i\ge 1$ the complex $Y^i$ is obtained from $Y^{i-1}$ by randomly selecting and adding external simplexes of dimension $i$, i.e. simplexes $\sigma$ with $\dim \sigma=i$ such that $\partial \sigma\subset Y^{i-1}$; each external simplex 
$\sigma$ is added at random, with probability $p_\sigma$, independently of other $i$-dimensional simplexes. 

Random simplicial complexes $Y\in \Omega_n^\ast$ with respect of the upper probability measures $\op$ can also be constructed inductively. The complex $Y$ admits a filtration by simplicial subcomplexes
$$Y_0\subset Y_1\subset Y_2\subset \dots\subset Y_{n-1} =Y$$ where $Y_0\subset \Delta_n$ is a pure $(n-1)$-dimensional random complex obtained by selecting each $(n-1)$-dimensional simplex $\sigma\subset [n]$ at random, with probability $p_\sigma$, independently of the other $(n-1)$-dimensional simplexes. $Y_0$ comprises of the union of the selected $(n-1)$-dimensional simplexes and all their faces.
Each following complex $Y_i$, $i\ge 1$,  is obtained from the previous one $Y_{i-1}$ by adding simplexes $\sigma$ of dimension $n-1-i$ which are not in $Y_{i-1}$; each such $\sigma$ is added at random, with probability $p_\sigma$, independently of the choices made with respect to the other $(n-1-i)$-dimensional simplexes. Once a simplex $\sigma$ of dimension $(n-1-i)$ is added to $Y_{i-1}$ all faces of $\sigma$ are automatically added so that $Y_i$ is a simplicial complex. 

\begin{definition} \label{homog} We shall say that a system of probability parameters $p_\sigma$ is homogeneous if $p_\sigma =p_{\sigma'}$ assuming that $\dim \sigma = \dim \sigma'$. 
\end{definition}

Note that most random complexes which appear in literature are homogeneous. For example the multi-parameter random simplicial complexes of \cite{costa}, \cite{costa1}, \cite{costa2}, \cite{costa3} are homogeneous lower random simplicial complexes. 

\subsection{Duality between the lower and upper models} In this subsection we review a result of \cite{farber} describing an Alexander type duality between the upper and lower models. 

Let $\partial \Delta_n$ denote the boundary of the $n$-dimensional simplex, viewed as a simplicial subcomplex. The set of vertices of $\partial \Delta_n$ is $[n]$ and the set of simplexes of $\partial\Delta_n$ is the set of all proper nonempty subsets $\sigma \subsetneq [n]$. 

For any simplex $\sigma\in \partial \Delta_n$ we denote by $\check \sigma$ the simplex spanned by the complementary set of vertices,
i.e. $V(\check \sigma) = [n]-V(\sigma)$. Clearly, $\dim \sigma +\dim \check \sigma = n-1.$

For a simplicial complex $X\subset \partial \Delta_n$ we denote by $\cc(X)$ the simplicial complex defined by the following rule:
\begin{eqnarray}
\sigma\in \cc(X) \iff \check \sigma \notin X. 
\end{eqnarray}

In particular a vertex $i\in [n]$ belongs to $\cc(X)$ 
iff the complementary $(n-1)$-dimensional face is not in $X$. If $\tau\subset \sigma$ then $\check \sigma \subset \check \tau$ and 
$\check \sigma \notin X$ implies $\check \tau\notin X$. This shows that $\cc(X)$ is a simplicial complex. 

As a remark we mention that $\dim X\le r$ if and only if the dual complex $\cc(X)$ contains full $(n-r-2)$-dimensional skeleton of $\Delta_n$.

Clearly the map $X\mapsto \cc(X)$ is involutive, i.e. $\cc(\cc(X))=X$. 

By Corollary 9.7 from \cite{farber}, the Betti numbers of $\cc(X)$ are given by the formulae
\begin{eqnarray}\label{betti}
b_j(\cc(X) )= b_{n-2-j}(X), 
\end{eqnarray}
for $j=0, 1, \dots, n-2. $
%Next we describe a duality relation between systems of probability parameters. 
Let $\sigma \mapsto p_\sigma$ be a system of probability parameters. {\it The dual system} $\sigma\mapsto p'_\sigma$ is defined  by 
\begin{eqnarray}\label{dualsys}
p'_\sigma= 1-p_{\check \sigma}.
\end{eqnarray}

\begin{theorem}\label{dual}
Let $p_\sigma$ be a system of probability parameters and let $p'_\sigma$ be the dual system. Consider the lower probability measure 
$\up$ on $\Omega_n^\ast$ with respect to the system $p_\sigma$. Besides, consider the upper probability measure $\op'$ on 
$\Omega_n^\ast$ with respect to the dual system $p'_\sigma$. Then the duality map $$X\mapsto \cc(X)$$ is an isomorphism of probability spaces $(\Omega_n^\ast, \up)\to (\Omega_n^\ast, \op')$. 
\end{theorem}

This statement is part of Proposition 9.8 from \cite{farber}. 

\subsection{Links as random complexes} The technical result of this subsection will be used later in this paper. It is a generalisation of Theorem 6.2 from \cite{farber} and Lemma 3.6 from \cite{costa1}.

Let $V\subset [n]$ be a fixed vertex set of cardinality $k$ and let $\Delta'$ denote the simplex spanned by the set $[n]-V$. 
We shall consider simplicial complexes $Y\in \Omega_n^\ast$ containing $V$. The link $\lk_Y(V)$ is defined as the 
union of all simplexes $\sigma$ which are disjoint from $V$ and such that the simplex $v\sigma$ is contained in $Y$ for any $v\in V$. 
The simplex $v\sigma$ is a cone over $\sigma$ with apex $v$. Clearly $\lk_Y(V)$ is a simplicial subcomplex of $\Delta'$.

\begin{lemma}\label{linkofkvert} 
Let $Y\in \Omega_n^\ast$ be a random simplicial complex with respect to the lower measure with probability parameters 
$\{p_\sigma\}$ containing the set of vertices $V$ and consider the link $\lk_Y(V)\subset \Delta'$ as a random simplicial subcomplex of $\Delta'$. Then $\lk_Y(V)\subset \Delta'$ is a random simplicial subcomplex with respect to the lower probability measure with the set of probability parameters 
$$p'_\tau = p_\tau\cdot\prod_{v\in V} p_{v\tau},$$
where $\tau$ is is a simplex in $\Delta'$. 
\end{lemma}
\begin{proof}
Define the following probability function on the set of all subcomplexes $L\subset\Delta'$
$$\underline{\lambda}(L)=\up(V\subset Y)^{-1}\cdot \sum_{V\subset Y \& \lk_Y(V) = L}\up(Y).$$
Here $\up(V\subset Y)^{-1}= \left(\prod_{v\in V} p_v\right)^{-1}$ is a normalising factor.
We want to compute probability that $\lk_Y(V)$ contains a given subcomplex $L\subset \Delta'$, i.e.
\begin{align*}
\underline{\lambda}(\lk_Y(V)\supset L) &= \up(V\subset Y)^{-1}\cdot \sum_{V\subset Y \& \lk_Y(V) \supset L}\up(Y)\\
&= \up\left(V\subseteq Y\right)^{-1}\cdot \up\left(V\ast L\subset Y\right)\\
&= \left(\prod_{v\in V}p_v\right)^{-1}\cdot \prod_{\sigma\in V\ast L} p_\sigma\\
&= \left(\prod_{v\in V}p_v\right)^{-1}\cdot \left(\prod_{v\in V}p_v\cdot \prod_{\tau\in L}p_\tau\cdot\prod_{v\in V, \tau\in L}p_{v\tau} \right)\\
&= \prod_{\tau\in L}\left[p_\tau\cdot\prod_{v\in V} p_{v\tau}\right] = \prod_{\tau\in L} p_\tau'.
\end{align*}
The statement of Lemma \ref{linkofkvert} now follows from the intrinsic characterisation of the lower probability measure given by  
Corollary 5.3 in \cite{farber}.
\end{proof}

\section{The medial regime. Statements of the main results}\label{sec2}

We shall say that the system of probability parameters $\{p_\sigma\}$ is {\it in the medial regime} if there exist constants 
$p,  \, P\in (0, 1)$ 
such that for any simplex $\sigma\in \Delta_n$ one has
\begin{eqnarray}\label{medial}
0<p\, \le\,  p_\sigma\, \le\,  P<1.
\end{eqnarray}
We emphasise that the numbers $p, \, P$ are supposed to be independent of $n$. In other words, in the medial regime the probability parameters $p_\sigma$ are allowed to approach neither 0 nor 1, as $n\to \infty$. It will be convenient to write
\begin{eqnarray}\label{medial1}
p=e^{-a}, \quad P=e^{-A}\end{eqnarray}
where the $0<A\le a$ are constants.

Next we state two main results of this paper:

\begin{theorem}\label{thm1} Let $Y\in \Omega_n^\ast$ be a random simplicial complex in the medial regime with respect to the lower measure. 
Then:
\begin{enumerate}
\item The dimension of $Y$ satisfies 
\begin{eqnarray*}
\lfloor \beta(n, a)\rfloor - 1\, \le\,  \dim Y\,  \le\,  \beta(n, A)-1 +\epsilon_0,
%\log_2\ln n +\log_2\log_2\ln n - \log_2 a -1 \le \dim Y \le \lceil \log_2\ln n +\log_2\log_2\ln n - \log_2 A - 1\rceil, 
\end{eqnarray*}
a.a.s. Here $\epsilon_0>0$ is an arbitrary positive constant and we use the notation 
%\begin{eqnarray*}\label{beta1}
$$\beta(n, y) = \log_2\ln n +\log_2\log_2\ln n - \log_2(y);$$ %\end{eqnarray*}
\item $Y$ is connected and simply connected, a.a.s;
\item If the system of probability parameters  $p_\sigma$ is homogeneous (see Definition \ref{homog}) then with probability tending to $1$ as $n\to \infty$ the Betti numbers $b_j(Y)$ vanish for all $$0< j \, \le  \, \log_2\ln n -\log_2 a-1 -\delta_0,$$ where $\delta_0>0$ is an arbitrary constant. 
\end{enumerate}
\end{theorem} 

Thus, under the assumptions of Theorem \ref{thm1} a random complex $Y$ may potentially have nontrivial reduced Betti numbers only in dimensions $j$ satisfying
$$\log_2\ln n  -\log_2a -1 -\delta_0 <j \le   \log_2\ln n +   \log_2\log_2\ln n -\log_2A -1 +\epsilon_0,$$
a.a.s.

To illustrate Theorem \ref{thm1}, let us assume that the integer $n$ is written in the form $n = e^{2^{k}}$. Then the dimension of the random 
complex $Y$ satisfies 
$$\dim Y \, \sim \, k+\log_2 k,$$
%$$\lfloor k +\log_2 k -\log_2 a\rfloor -1\le \dim Y \le k+\log_2 k-\log_2 A -1 +\epsilon_0$$
and the range of potentially nontrivial Betti numbers is roughly 
$$k\le j \le k+\log_2 k.$$
%$$k-\log_2 a -1-\delta_0<j \le k+\log_2k -\log_2A -1 +\epsilon_0.$$
We see that a lower model random simplicial complex in the medial regime is homologically highly connected with nontrivial Betti numbers concentrated in a thin layer of dimensions near the dimension of the complex.

In the following Theorem we shall describe the properties of the random simplicial complexes in the upper model. If the initial system of probability parameters $p_\sigma$ is in a medial regime (\ref{medial}) then the dual system $p'_\sigma$ (see (\ref{dualsys})) will also be in the medial regime since
$$0<1-P \le p'_\sigma \le 1-p<1.$$
We shall need the 
 {\it dual numbers} 
$$0<a'\le A'$$ defined by the equations
\begin{eqnarray}\label{prime}
e^{-a}+e^{-a'} =1=e^{-A}+e^{-A'}.\end{eqnarray}
One has 
$e^{-A'}\le p'_\sigma \le e^{-a'}.$

\begin{theorem}\label{thm2} Let $Y\in \Omega_n^\ast$ be a random simplicial complex with respect to the upper probability measure 
associated to a system of probability parameters $p_\sigma$. Assume that $p_\sigma$ satisfies $$0<p\le p_\sigma\le P<1,$$ 
where $p=e^{-a}$ and $P=e^{-A}$ are constant, i.e. the system of probability parameters is in the medial regime.
Then, with probability tending to $1$, one has:
\begin{enumerate}
\item The dimension $\dim Y$ equals $n-1$;
\item The maximal dimension $d$ such that $Y$ contains the $(n-d)$-dimensional skeleton $\Delta_n^{(n-d)}$ of the simplex $\Delta_n$ 
satisfies 
$$\lfloor \beta(n, A')\rfloor +1 \, \le\,  d\le \,  \beta(n, a') +1+\epsilon_0.$$
Here $\epsilon_0>0$ is an arbitratry positive constant. 
\item If the system of probability parameters $p_\sigma$ is homogeneous then the reduced Betti numbers $b_j(Y)$ vanish for all dimensions $j$ except possibly
$$\log_2\ln n  -\log_2A' +1 -\delta_0 \, <\, n-j \, \le\,    \beta(n, a') +1+\epsilon_0.$$

%$$\log_2\ln n  -\log_2A' +1 -\delta_0 \, <\, n-j\, \le   \log_2\ln n +   \log_2\log_2\ln n -\log_2a' +1 +\epsilon_0,$$

%$$n-\lceil \beta(n, A')\rceil \, \le\,  j \, \le\,  n-2 - \lfloor \log_2\ln n - \log_2 a'\rfloor.$$

\end{enumerate}
\end{theorem}

We see that the topology of a typical random simplicial complex $Y$ in the upper model in the medial regime is totally different from one in the lower model. If $n$ is written in the form $n=e^{2^k}$ then $Y$ contains the skeleton $\Delta_n^{(n-d)}$, where 
$$d\sim k +\log_2k$$
and the nontrivial Betti numbers of $Y$ are concentrated in an interval of dimensions of width $\sim \log_2 k$ above the dimension 
$n-d\, \sim \, n-k -\log_2 k$. 

\begin{remark} {\rm Statement (3) of Theorem \ref{thm1} and statement (3) of Theorem \ref{thm2} require the system of probability parameters $p_\sigma$ to be homogeneous. We believe that this assumption is unnecessary. The proofs presented here use a concentration result for the spectral gap of Erd\H os - R\'enyi random graphs;  we are not aware of a more general result of this type applicable to inhomogeneous random graphs. }
\end{remark}

The proofs of Theorems \ref{thm1} and \ref{thm2} are given is the following sections.  Theorem \ref{thm1} is the summary 
of Proposition \ref{prop1}, Corollary \ref{corconn}, Proposition \ref{propsimplyconn} and Theorem \ref{thmtrivialhomology}. 
The proof of Theorem \ref{thm2} is given in section \S \ref{prfthm2}. 

\section{Coupling}

Recall that the upper and lower models of random simplicial complexes depend on the choice of 
a function
associating with each simplex $\sigma\subset [n]$ a
probability parameter $p_\sigma\in [0,1]$.  In this section we compare the properties of random simplicial complexes 
$Y$ and $Y'$ in two models having different probability parameters $p_\sigma$ and $p'_\sigma$. 
We show that for $p_\sigma\le p'_\sigma$ one may {\it \lq\lq realise\rq\rq}\,  $Y$ as a subcomplex of $Y'$. 
This leads to the conclusion that for any monotone property $\mathcal P$ of random simplicial complexes the probability of the event 
$Y\in \mathcal P$ is dominated by the probability of the event $Y'\in \mathcal P$. 

Next we introduce some notations. We denote by $\up$ and $\up'$ the lower probability measures on the set $\Omega_n$ of random simplicial complexes $Y\subset \Delta_n$ associated to the systems of probability parameters $p_\sigma$ and $p'_\sigma$ correspondingly. 
We shall denote by $\op$ and $\op'$ the corresponding upper measures on $\Omega_n$. Consider also the set $\mathfrak P\Omega_n$ of all pairs
$(X, Y)$ consisting of a simplicial complex $X\subset \Delta_n$ and its subcomplex $Y\subset X$. There are two projections
$$\pi_1, \, \pi_2: \, \mathfrak P\Omega_n\to \Omega_n$$
where $\pi_1(X,Y)=X$ and $\pi_2(X,Y)=Y$. 

\begin{theorem}\label{coupling} (A) Suppose that two systems of probability parameters $p_\sigma \le p'_\sigma$ are given. Then there exists a probability measure $\underline \mu$ on $\mathfrak P\Omega_n$ such that its direct images under the projections $\pi_1, \pi_2$ are 
\begin{eqnarray}\label{ten}
(\pi_1)_\ast(\underline \mu)= \up'\quad \mbox{and}\quad 
(\pi_2)_\ast(\underline \mu)= \up.
\end{eqnarray} 
Similarly, there exists a probability measure $\overline \mu$ on $\mathfrak P\Omega_n$ such that its direct images under the projections $\pi_1, \pi_2$ are 
\begin{eqnarray}\label{eleven}
(\pi_1)_\ast(\overline \mu)= \op', \quad
(\pi_2)_\ast(\overline \mu)= \op. 
\end{eqnarray}
(B) Suppose additionally that $p_\sigma = p'_\sigma$ for any simplex $\sigma$ of dimension $\le k$, where $k\ge 0$ is an integer. Then the measure $\underline \mu$ on $\mathfrak P\Omega_n$ is supported on the sets of pairs $(X, Y)$ of simplicial complexes having identical $k$-dimensional skeleta, i.e. $X^{(k)}=Y^{(k)}$. 

(C) If $p_\sigma=p'_\sigma$ for all simplexes $\sigma$ of dimension $> k$ where $k$ is fixed integer) then the measure $\overline \mu$ is supported on the sets of pairs $(X, Y)$ of simplicial complexes satisfying $X-X^{(k)}=Y-Y^{(k)}$. 
\end{theorem}

Let $\mathcal P$ be a property of a simplicial complex which is monotone, i.e. 
 $Y\in \mathcal P$ implies $X\in \mathcal P$ for a simplicial subcomplex $Y\subset X$. 
\begin{corollary}
Under the assumption $p_\sigma\le p'_\sigma$, for any monotone property $\mathcal P$ one has
\begin{eqnarray}
\up(Y\in \mathcal P) \le \up'(Y\in \mathcal P)\quad\mbox{and}\quad \op(Y\in \mathcal P) \le \op'(Y\in \mathcal P). 
\end{eqnarray}
\end{corollary}
\begin{proof} Applying Theorem \ref{coupling} one has
$$\up(Y\in \mathcal P) = \underline \mu(\left\{(X, Y); Y\in \mathcal P\right\})\le \underline \mu(\left\{(X, Y); X\in \mathcal P\right\})= \up'(X\in \mathcal P).$$
The case of the upper measure $\overline \mu$ is similar. 
\end{proof}

As an example we consider the property $\dim Y\ge d$ where $d$ is an integer. Since it is monotone we obtain:
\begin{corollary}\label{cor531} Under the assumption that $p_\sigma\le p'_\sigma$ for every simplex $\sigma\subset [n]$, 
one has 
$$\up(\dim Y\ge d) \le \up'(\dim Y\ge d) \quad \mbox{and}\quad \op(\dim Y\ge d) \le \op'(\dim Y\ge d)$$
for any integer $d\ge 0$.  Here $\up$ and $\up'$ are lower probability measures on $\Omega_n$ associated to the systems of probability parameters 
$p_\sigma$ and $p'_\sigma$, correspondingly.
\end{corollary}

The following arguments will be used in the proof of Theorem \ref{coupling}. 

Let $S$ be a finite set and suppose that for each element $s\in S$ we are given a probability parameter $p_s\in [0,1]$. {\it The Bernoulli measure} $\nu$ on the set $2^S$ of all subsets of $S$ is characterised by the property that for $A\subset S$ one has 
\begin{eqnarray}\label{ber1}
\nu(A) = \prod_{s\in A} p_s \cdot \prod_{s\notin A} (1-p_s).
\end{eqnarray}
%$$\nu(\left\{X\subset S; A\subset X\right\}=\prod_{s\in A} p_s. $$
Consider now another set of probability parameters $p'_s\in [0,1]$ with the property 
$$p_s\le p'_s$$ for any $s\in S$; let $\nu'$ be the corresponding Bernoulli measure on $2^S$, i.e. 
\begin{eqnarray}\label{ber2}
\nu'(A) = \prod_{s\in A} p'_s \cdot \prod_{s\notin A} (1-p'_s).
\end{eqnarray}
%$$\nu'(\left\{X\subset S; A\subset X\right\}=\prod_{s\in A} p'_s. $$
\begin{lemma}\label{lmcoupling}
Let $\mathfrak P\Omega_S$ denote the set of all pairs $(X, Y)$ where $Y\subset X\subset S$. Consider the projections 
$\pi_1, \pi_2: \mathfrak P\Omega_S \to 2^S$ where $\pi_1(X, Y)=X$ and $\pi_2(X,Y)=Y$.
There exists a probability measure $\mu$ on $\mathfrak P\Omega_S$ such that 
\begin{eqnarray}\label{marginals}
(\pi_1)_\ast(\mu) = \nu' \quad\mbox{and}\quad (\pi_2)_\ast(\mu) = \nu.
\end{eqnarray}
If $p_s=p'_s$ for all elements $s$ in a subset $ T\subset S$ then the measure $\mu$ is supported on the set of pairs $(X,Y)$ of subsets of $S$ satisfying $X\cap T= Y\cap T$. 
\end{lemma} 
\begin{proof}  We define a probability measure $\mu$ on $\mathfrak P\Omega_S$ by the formula:
\begin{eqnarray}\label{ber3}
\mu(X,Y) = \prod_{s\in X} p_s'\cdot \prod_{s\in S- X}(1-p_s')\cdot 
\prod_{s\in Y} \frac{p_s}{p'_s} \cdot \prod_{s\in X-Y} \left(1-\frac{p_s}{p'_s}\right).
\end{eqnarray}
The equalities (\ref{marginals}) can be verified directly. The assumption $p_s\le p'_s$ is used to ensure non-negativity of $\mu$. 
If there exists an element $s\in X-Y$ which lies in $T$, then $p_s=p'_s$ and $\mu(X,Y)=0$ since the last factor in (\ref{ber3}) vanishes. 
\end{proof}

\begin{proof}[Proof of Theorem \ref{coupling}] We apply Lemma \ref{lmcoupling} with $S=2^{[n]}$, the set of subsets of the set of vertices $[n]$. The subsets $X\subset S$ can be identified with hypergraphs and we see that the set $\Omega_S=2^S$ is the set of all hypergraphs with vertices in $[n]$ which in Section \ref{sec:twomodels} was denoted $\Omega_n$. The two systems of probability parameters $p_\sigma$ and $p'_\sigma$ (where $\sigma\in S$ is a simplex) define two Bernoulli probability measures on 
$2^S=\Omega_S=\Omega_n$ which we shall denote by  $\nu$ and $\nu'$ correspondingly, see formulae (\ref{ber1}) and (\ref{ber2}). 

The set of pairs $\mathfrak P\Omega_S$ which appears in Lemma \ref{lmcoupling} can be viewed as the set of pairs of hypergraphs
$(X, Y)$ where $Y$ is a subhypergraph of $X$. Since for any simplex $\sigma$ one has $p_\sigma\le p'_\sigma$, we may apply Lemma 
\ref{lmcoupling} to obtain a probability measure $\mu$ on $\mathfrak P\Omega_S=\mathfrak P\Omega_n$ with the property $(\pi_1)_\ast(\mu)=\nu'$ and $(\pi_2)_\ast(\mu)=\nu$.

Consider the maps $\underline r, \overline r: \Omega_n\to \Omega_n^\ast$ (see (\ref{maps}) in \S \ref{sec:twomodels}) where $\Omega_n^\ast$ denotes the set of all simplicial subcomplexes of $\Delta_n$.  These maps obviously define maps of pairs
$\underline r, \overline r: \mathfrak P\Omega_n\to \mathfrak P\Omega_n^\ast$ and we define the probability measures $\underline \mu, \overline\mu$ on $\mathfrak P\Omega_n^\ast$ by the formulae
\begin{eqnarray}
\underline \mu = (\underline r)_\ast(\mu), \quad \overline \mu = (\overline r)_\ast(\mu).
\end{eqnarray}
We have two commutative diagrams 
$$\xymatrix{\mathfrak P\Omega_n \ar[r]^{\pi_i}\ar[d]^-{\underline r} &\Omega_n\ar[d]^{\underline r}\\
\mathfrak P\Omega^\ast_n \ar[r]^-{\pi_i} &\Omega_n^\ast,
}
\quad \quad
\xymatrix{\mathfrak P\Omega_n \ar[r]^{\pi_i}\ar[d]^-{\overline r} &\Omega_n\ar[d]^{\overline r}\\
\mathfrak P\Omega^\ast_n \ar[r]^-{\pi_i} &\Omega_n^\ast.
}
$$
where $i=1, 2$. 
%$$\xymatrix{
%X^I \ar[rr]^-p \ar[dr]_{\pi_X} & & \tilde X\times_\pi \tilde X \ar[dl]^-{Q} \\
%& X \times X &
%}
%$$
Applying the definitions, we obtain
$${\pi_1}_\ast(\underline\mu) = {\pi_1}_\ast(\underline r_\ast(\mu))={ \underline r}_\ast({\pi_1}_\ast(\mu) = { \underline r}_\ast(\nu')= \up'.$$
And similarly
$${\pi_2}_\ast(\underline\mu) = {\pi_2}_\ast(\underline r_\ast(\mu))={ \underline r}_\ast({\pi_2}_\ast(\mu) = { \underline r}_\ast(\nu)= \up.$$
This proves formulae (\ref{ten}). Formulae (\ref{eleven}) follow similarly. This proves statement (A). 

To prove statement (B) we engage the last statement of Lemma \ref{lmcoupling} which claims that the constructed measure $\mu$ on $\mathfrak P\Omega_n$ is supported on the set of pairs of hypergraphs $(X, Y)$ having identical $k$-dimensional skeleta. Then obviously 
the measure $\underline \mu =(\underline r)_\ast(\mu)$ is supported on the set of pairs of simplicial complexes having identical $k$-skeleta.

The proof of (C) is similar. If $p_\sigma=p'_\sigma$ for any simplex of dimension greater than $k$ then the measure $\mu$ is supported on the set of pairs of hypegraphs $(X, Y) \in {\mathfrak P}\Omega_n$ which are identical above dimension $k$. This implies that the
direct image measure $\overline \mu =(\overline r)_\ast(\mu)$ is supported on the set of pairs of simplicial complexes which are identical above dimension $k$. 
\end{proof}

\section{Dimension of a lower random simplicial complex in the medial regime}

In this section we shall consider a random simplicial complex $Y\in \Omega_n^\ast$ with respect to the lower model and will impose the medial regime assumptions (\ref{medial}). We shall write 
\begin{eqnarray}
p=e^{-a}, \quad P=e^{-A} \quad \mbox{where}\quad 0<A\le a.
\end{eqnarray}
Let  us denote
\begin{eqnarray}\label{beta}
\beta=\beta(n, y) = \log_2\ln n +\log_2\log_2\ln n - \log_2(y). \end{eqnarray}
\begin{prop}\label{prop1} 
Let $\epsilon_0>0$ be a fixed constant. 
Under the above assumptions the dimension of a random simplicial complex $Y$ satisfies
\begin{eqnarray}\label{dim}
\lfloor\beta(n, a) \rfloor-1 \,  \le\,  \dim Y\,  \le\,   \beta(n, A)-1 +\epsilon_0 ,
\end{eqnarray}
a.a.s.
\end{prop}
\begin{remark}{\rm 
Note that the quantity 
\begin{eqnarray}\label{error}
\beta(n,A)-\beta(n, a)= \log_2\left(\frac{a}{A}\right)=\log_2\left(\frac{\ln p}{\ln P}\right)\ge 0
\end{eqnarray}
is constant (independent of $n$). Hence Proposition \ref{prop1} determines the dimension of a random complex 
$Y$ with finite error (\ref{error}) while the dimension itself $\dim Y$ tends to infinity. 

In the special case when $p=P$ and $a=A$ we obtain
$\lfloor \beta(n, a)\rfloor-1 \le \dim Y \le  \beta(n, a)-1+\epsilon_0,$ a.a.s. which 
nearly uniquely determines the dimension $\dim Y$.
% with possibly an error $1$, 
%$\dim Y= \lfloor \beta(n, a) -1\rfloor,$ a.a.s.
}
\end{remark}

\begin{proof}[Proof of Proposition \ref{prop1}] We start by establishing the upper bound in (\ref{dim}). Using the monotonicity of dimension we may apply Theorem \ref{coupling} and Corollary \ref{cor531}. 
Therefore in the proof of the upper bound we may assume without loss of generality that $$p_\sigma=P=e^{-A}$$ for any simplex $\sigma$. 

Let $f_\ell: \Omega_n^\ast\to \R$ denote the number of $\ell$-dimensional simplexes in $Y$. Note that as a random variable, $f_\ell=\sum X_\sigma$, where the sum runs over all simplexes $\sigma\subset [n]$ of dimension $\ell$ and 
$X_\sigma$ is a random variable which takes values $0$ and $1$ depending on whether the simplex $\sigma$ is included into the random complex $Y$. We have 
$$\E(X_\sigma) = \prod_{\nu\subset \sigma}p_\nu = P^{2^{\ell+1}-1}.$$
Then 
$$\E(f_\ell) =\binom {n+1}{\ell+1} \cdot P^{2^{\ell+1}-1}.$$
We may estimate the expectation from above as follows
$$\E(f_\ell) \le (n+1)^{\ell+1}\cdot P^{2^{\ell+1}-1}\le \frac{e}{P}  \cdot \left(\exp\left[\ln n - A\cdot \frac{2^{\ell+1}}{\ell+1}\right]\right)^{{\ell+1}}.$$
Since the function $x\mapsto \frac{2^x}{x}$ is monotone increasing for $x\ge 2$ we obtain that for any 
$$\ell\ge \beta(n,A)+\epsilon_0-1=\beta+\epsilon_0 -1$$ 
(where $\epsilon_0>0$ is fixed) one has
$$\ln n - A\frac{2^{\ell+1}}{\ell+1} \le \ln n - A\frac{2^{\beta+\epsilon_0}}{\beta+\epsilon_0}= - (2^{\epsilon_0} -1) \cdot \ln n \cdot \frac{\log_2\ln n }{\beta+\epsilon_0} \le -\frac{1}{2} (2^{\epsilon_0}-1)\cdot \ln n$$
implying 
$$\E(f_\ell) \le \frac{e}{P} n^{-c(\ell+1)}, \quad \mbox{where}\quad c= \frac{1}{2}\cdot (2^{\epsilon_0}-1)>0.$$
We obtain
$$
\sum_{\ell+1\ge \beta+\epsilon_0} \E(f_\ell) \le \frac{e}{P} \sum_{\ell+1\ge \beta+\epsilon_0} n^{-c(\ell+1)} \le \frac{e}{P}
\cdot \frac{n^{-c(\beta +\epsilon_0)}}{1-n^{-c}} \, \to\,  0.
$$

Thus, by the first moment method, $Y$ has no simplexes in any dimension $\ell\ge \beta+\epsilon_0-1$, a.a.s. i.e. we obtain the right inequality in (\ref{dim}). 

Next we prove the left inequality in (\ref{dim}), i.e. the lower bound for the dimension. While doing so we may assume (using Theorem \ref{coupling} and the monotonicity of dimension) that 
$$p_\sigma=p=e^{-a}$$
for any simplex $\sigma$. We assume below that 
\begin{eqnarray}\label{assumption}
\ell\le \beta(n, a)-1
\end{eqnarray} 
and our goal is to show that $f_\ell>0$ with probability tending to $1$ as $n\to \infty$. 
We shall use the following estimates for the binomial coefficient
\begin{eqnarray}
\frac{1}{3} 
\left(\frac{ne}{\ell}\right)^\ell\cdot \ell^{-1/2}\, \le \, 
\binom n \ell \, \le\,  \left(\frac{ne}{\ell}\right)^\ell\cdot \ell^{-1/2}
\end{eqnarray}
which are valid for $1\le \ell<n/2$ and $n$ large enough; it follows from Stirling's formula, see page 4 in \cite{Bo}. 
Hence we obtain 
\begin{eqnarray}
\E(f_\ell) = \binom {n+1}{\ell+1} p^{2^{\ell+1}-1} \ge \left(\frac{n}{\ell}\right)^{\ell+1} p^{2^{\ell+1}-1}= 
p^{-1}\left[\exp\left(\ln\frac{n}{\ell} -a \frac{2^{\ell+1}}{\ell+1}\right)
\right]^{\ell+1}.
\end{eqnarray}
Using (\ref{assumption}) we find that 
$$\frac{a2^{\ell+1}}{\ell+1} \le \ln n \cdot \frac{ \log_2\ln n}{\log_2\ln n +\log_2\log_2\ln n - \log_2a}$$
implying 
$$\ln\frac{n}{\ell} -a \frac{2^{\ell+1}}{\ell+1}\ge \frac{\ln n \cdot \log_2\log_2\ln n}{2 \cdot \log_2\ln n}.
$$
This shows that $\E(f_\ell)\to \infty$. 

We shall use the inequality 
\begin{eqnarray}
\Bbb P(f_\ell>0) \ge \frac{\E(f_\ell)^2}{\E(f_\ell^2)},
\end{eqnarray}
(see p. 54 of \cite{JLR}) and show that for 
under the assumptions (\ref{assumption})
the inverse quantity $\frac{\E(f_\ell^2)}{\E(f_\ell)^2}$ tends to $1$ as 
$ n\to \infty$. Since we know apriori that $\frac{\E(f_\ell^2)}{\E(f_\ell)^2}\ge 1$, it is enough to show that the ratio $\frac{\E(f_\ell^2)}{\E(f_\ell)^2}$ is bounded above by a sequence tending to $1$ as $n\to \infty$. 

As above, $f_\ell=\sum X_\sigma$, where the sum runs over all simplexes $\sigma\subset [n]$ of dimension $\ell$.
Hence $f_\ell^2 = \sum_{\sigma, \tau}X_\sigma X_\tau$ and 
$\E(f_\ell^2) = \sum_{\sigma, \tau}\E(X_\sigma X_\tau)$. We have
$$\E(X_\sigma X_\tau) = \up (\sigma\subset Y\, \&\, \tau\subset Y) = p^{2\cdot 2^{\ell+1}- 2^i-1}$$
where $i$ denotes the cardinality of intersection $\sigma\cap \tau\subset [n]$. 
One therefore obtains
$$\label{eqnexpsq}
\E\left(f_\ell^2\right) 
=\sum_{i=0}^{\ell +1}\binom{n+1}{\ell +1}\cdot \binom{\ell +1}{i}\cdot \binom{n-\ell}{\ell +1-i}\cdot p^{2\cdot 2^{\ell+1} - 2^i - 1}
$$
and since $$\E(f_\ell) =\binom {n+1}{\ell+1} p^{2^{\ell+1}-1}$$ we obtain
$$
\dfrac{\E(f_\ell^2)}{\E(f_\ell)^2} = \sum_{i=0}^{\ell+1} \dfrac{\binom{\ell+1} i\cdot \binom {n-\ell}{\ell+1-i}}{\binom {n+1}{\ell+1}}\cdot p^{-2^i+1}.
$$
We shall denote by $r_i$ the terms in the last sums where $i=0, 1, \dots, \ell+1$. For the term $r_0$ we have 
$$r_0=\frac{\binom{n-\ell}{\ell+1}}{\binom{n+1}{\ell+1}}<1.$$
One goal is to show that the sum of all other terms $r_1+r_2+\dots+r_{\ell+1}$ tends to zero with $n$. 
For the term $r_1$ we have 
$$r_1= \dfrac{(\ell+1)\binom{n-\ell} \ell}{\binom{n+1}{\ell+1}} \cdot p^{-1} \le \dfrac{(\ell+1)\cdot n^\ell}{\left(\dfrac{n}{\ell+1}\right)^{\ell+1}}\cdot p^{-1} =
\dfrac{(\ell+1)^{\ell+2}}{n}\cdot p^{-1}. $$
Using our assumption (\ref{assumption}) and (\ref{beta}) we see that $r_1\to 0$ as $n\to \infty$.  

Next we consider the term $r_i$ with $2\le i \le \ell+1$. Since $p^{-1}=e^a$ and taking into account that the function $\dfrac{2^x}{x}$ is increasing for $x\ge 2$ we obtain
\begin{eqnarray*}
r_i&=& \dfrac{\binom{\ell+1} i\cdot \binom{n-\ell}{\ell+1-i}}{\binom{n+1}{\ell+1}}\cdot p^{-2^{i}+1}\le \frac{(\ell+1)^{\ell+i+1}}{n^i}\cdot p^{-2^i}
\\
&\le &
\beta^{2\beta}\cdot\left\{\exp\left[\dfrac{a2^i}{i}-\ln n\right]\right\}^i
\le \beta^{2\beta}\cdot\left\{\exp\left[\dfrac{a2^\beta}{\beta}-\ln n\right]\right\}^i,
\end{eqnarray*}
where we have used (\ref{assumption}) and the following standard inequalities for the binomial coefficients $$\dfrac{a^b}{b^b}\leq \binom{a}{b}\leq a^b.$$ 
One has 
$$\dfrac{a2^\beta}{\beta} -\ln n = - \ln n \cdot \dfrac{\log_2\log_2\ln n -\log_2 a}{\beta}\le - \ln n \cdot \dfrac{\log_2\log_2\ln n}{2\cdot\log_2\ln n}.$$
Denoting 
$$\gamma=\gamma(n) = \dfrac{\log_2\log_2\ln n}{2\cdot\log_2\ln n}$$
we may write, for $i\ge 2$,
$$r_i \le \beta^{2\beta} \cdot \left\{\exp(-\gamma\cdot \ln n)\right\}^i = \dfrac{\beta^{2\beta}}{n^{i\gamma}}.$$
Clearly, $\gamma\to 0$. Summing up we obtain
$$\sum_{i=2}^{\ell+1} r_i \le \beta^{2\beta}\cdot \dfrac{n^{-2\gamma}}{1- n^{-\gamma}}.$$
The lower bound estimate in (\ref{dim})  would now follow once we know that $n^\gamma\to \infty$ and moreover $\dfrac{\beta^{\beta}}{n^{\gamma}}\to 0$. 
This is equivalent to  $$\beta\cdot \log_2 \beta - \gamma \cdot\ln n \to -\infty.$$ Since $\beta< 2 \log_2\ln n$ it is sufficient to show that 
$$2\log_2\ln n \cdot \log_2(2\log_2\ln n)-\ln n\cdot \dfrac{\log_2\log_2\ln n}{2\log_2\ln n}\to -\infty.$$
The above expression can be written in the form 
\begin{eqnarray}\label{expr}
2\log_2\ln n\cdot\left[1+\log_2\log_2\ln n\cdot \left[1- \dfrac{\ln n}{4(\log_2\ln n)^2}\right]\right]\end{eqnarray}
and since obviously $\dfrac{\ln n}{(\log_2\ln n)^2}\to \infty$, we see that (\ref{expr}) tends to $-\infty$. 

This completes the proof of Proposition \ref{prop1}. 
\end{proof}

\section{Simple connectivity of lower random simplicial complex in the medial regime}

In order to establish connectivity and simple connectivity of random simplicial complex in the medial regime we shall consider the cover by closed stars of vertices and apply the nerve lemma. Recall that the lower probability $\up$ of a random simplicial complex 
$Y\in \Omega_n^\ast$ is given by (\ref{explicitly}) and the medial regime assumptions are (\ref{medial}), see also (\ref{medial1}). 

\subsection{Common neighbours} Recall that {\it a common neighbour} of a set $S\subset Y$ of vertices in a simplicial complex $Y$ is a vertex $v\in Y -  S$ which is connected by an edge to every vertex of $S$. 
%We shall also consider a vertex $v\in S$ being a common neighbour for $S$ if it is connected by an edge in $Y$ to all other vertexes of $S$. 

\begin{lemma}\label{common} Let $0< \epsilon\le 1$ be fixed. Let $Y\in \Omega_n^\ast$ be a random simplicial complex with respect to the lower measure in the medial regime. Then any set $S$ of 
$$\left\lfloor \frac{\ln n}{(1+\epsilon)a}\right\rfloor$$ vertices of $Y$ have a common neighbour with probability at least 
$1- C\cdot \exp (-\frac{n^{\epsilon/2}}{2}).$ Here $C>0$ is a constant independent of $n$ (which however depends on the value of $p$). 
\end{lemma}
The number $a$ which appears in the statement is defined in (\ref{medial1}). 
\begin{proof}
Let $S\subset Y$ be a set of $k$ vertices. A vertex $v\not\in S$ is a common neighbour for $S$ with probability 
$p_v\cdot \prod_{u\in S} p_{(uv)}.$ 
%
%A vertex $v\in S\subset Y$ is a common neighbour for $S$ with probability
%$\prod_{u\in S-\{v\}} p_{(uv)}.$
Hence, a set $S\subset Y$ has no common neighbours in $Y-S$ with probability
$$\prod_{v\notin S} \left(1-p_v\cdot 
\prod_{u\in S} p_{(uv)}
%\cdot \prod_{u\in S-\{v\}} p_{(uv)}
\right)
\le \left(1-p^{k+1}\right)^{n+1-k}.$$
Let $X_k:\Omega_n^\ast \to \Z$ be the random variable counting the number of $k$ element subsets $S\subset Y$ having no common neighbours in $Y-S$.
Using the above inequality, we see that the expectation $\E(X_k)$  
%number of $k$ element subsets $S\subset Y$ having no common neighbours in $Y-S$ 
is bounded above by
\begin{align*}
\binom{n+1}{k}\cdot \left(1- p^{k+1}\right)^{n+1-k}&\leq n^k\cdot \exp\left(-(n+1-k)\cdot p^{k+1}\right)\\
&= \exp\left(k\ln n - (n+1-k)\cdot p^{k+1}\right)\\
&\leq C\cdot \exp\left(k\ln n - n\cdot p^{k+1}\right).
\end{align*}
In the final line we have used the fact that $(k-1) p^{k}$ is bounded for any $k\geq 2$.
For $n$ fixed the function $k\mapsto k\ln n  - n\cdot p^{k+1}$ is monotone increasing. Using this we find that for $ k\leq\frac{\ln n}{(1+\epsilon)a}$ 
%the probability that $Y$ has $k$ vertices without a common neighbour is at most
$$\E(X_k)\le C\cdot \exp \left(\frac{(\ln n)^2}{(1+\epsilon) a} - n\cdot e^{-a\left(\frac{\ln n}{(1+\epsilon)a}\right)}\right) = 
C\cdot \exp \left(\frac{(\ln n)^2}{(1+\epsilon) a} - n^{\frac{\epsilon}{1+\epsilon}}\right)\le C\exp(-\frac{n^{\epsilon/2}}{2}).
$$
Hence we obtain $$\up(X_k>0) \, \le\,  \E(X_k) \, \le\,  C\exp(-\frac{n^{\epsilon/2}}{2}).$$
This completes the proof.
\end{proof}

\begin{corollary}\label{corconn}
Let $Y\in \Omega_n^\ast$ be a random simplicial complex with respect to the lower measure in the medial regime.
Then the complex $Y$ is connected with probability at least
$$1 - C\exp\left(-\dfrac{n^{1/2}}{2}\right),$$
where $C > 0$ is a constant depending on $p$ and independent of $n$. 
\end{corollary}
\begin{proof}
Applying Lemma \ref{common} with $\epsilon =1$ we obtain that
any two vertices of $Y$ have a common neighbour in $Y$ with probability at least 
$1-C\exp(-\frac{n^{1/2}}{2})$. 
Then obviously any two vertices can be connected by a path in $Y$, i.e. $Y$ is path-connected with probability at least 
$1-C\exp(-\frac{n^{1/2}}{2})$. 
\end{proof}

\subsection{Simple connectivity}
Recall that a simplicial complex $X$ is said to be simply connected if it is connected and its fundamental group $\pi_1(X,x_0)$ is trivial. 
Our goal is to prove the following statement:
\begin{prop}\label{propsimplyconn}
A random simplicial complex $Y\in \Omega_n^\ast$ with respect to the lower probability measure in the medial regime 
is simply connected, a.a.s.
\end{prop}
The proof will consist of applying the Nerve Lemma (see \cite{bjornertop}, Theorem 10.6) to the cover $\U$ of $Y$ formed by the closed 
stars of vertexes. Recall that for a vertex $v\in Y$ the closed star $\St(v)\subset Y$ is the union of all closed simplexes $\sigma\in Y$ such that $v\in \sigma$. 
The nerve $\N(\U)$ of this cover is the simplicial complex with the vertex set identical to the vertex set of $Y$ and a set $S$ of vertices of $Y$ forms a simplex in $\N(\U)$ iff the intersection 
\begin{eqnarray}\label{intersection}
\cap_{v\in S}\St(v) \not=\emptyset
\end{eqnarray}
is not empty. Note that this intersection (\ref{intersection}) is not empty if the set of vertexes $S$ has a common neighbour. 
Rephrasing Lemma \ref{common} we obtain:

\begin{corollary}\label{nerve1}
Let $Y\in \Omega_n^\ast$ be a random simplicial complex with respect to the lower probability measure in the medial regime. Let $\U$ denote the cover of $Y$ formed by the closed stars of vertexes of $Y$. Then for any constant $0<\alpha<1$, the 
nerve complex $\N(\U)$ contains the full $\left\lfloor \alpha\cdot \log_{(p^{-1})} n\right\rfloor$-dimensional skeleton of the simplex spanned by the vertex set of $Y$. In particular, the nerve complex $\N(\U)$ is $\left(\left\lfloor \alpha\cdot \log_{(p^{-1})} n\right\rfloor-1\right)$-connected, a.a.s.
\end{corollary} 
Recall that the parameter $0<p<1$ of Lemma \ref{nerve1} is the one which appears in the definition of the medial regime, see (\ref{medial}). 

\begin{proof}[Proof of Proposition \ref{propsimplyconn}]
First we recall the Nerve Lemma, see \cite{bjornertop}, Theorem 10.6:

\begin{lemma}%{\rm{(Nerve Lemma)}.}\
\label{lemnerve}
If $Y$ is a simplicial complex and $\{S_i\}_{i\in I}$ is a family of subcomplexes covering $Y$ such that for any $t\geq 1$ every non-empty intersection
$S_{i_1}\cap\dots\cap S_{i_t}$
is $(k-t+1)$-connected.  Then $Y$ is $k$-connected if and only if the nerve complex $\mathcal{N}(\{S_i\}_{i\in I})$ is $k$-connected.
\end{lemma}

To prove Proposition \ref{propsimplyconn} we shall apply Lemma \ref{lemnerve} with $k=1$ to the cover $\{\St(v)\}$ of $Y$ formed by closed stars of vertexes $v\in Y$. Each of the stars $\St(v)$ is contractible and the nerve complex $\N(\{\St(v)\})$ is simply connected (see Corollary \ref{nerve1}), a.a.s. To complete the proof we need to show that any nonempty intersection $\St(v)\cap \St(w)$ is connected, a.a.s.

Note that
\begin{eqnarray}
\St(v)\cap \St(w) = \left\{
\begin{array}{lll}
\lk_Y(v) \cap \lk_Y(w), &\mbox{if} & (vw)\notin Y,\\
(\lk_Y(v) \cap \lk_Y(w))\cup \St(vw), &\mbox{if} & (vw)\in Y.
\end{array}
\right. 
\end{eqnarray}
Here $(vw)$ denotes the edge connecting $v$ and $w$. 

We shall denote by $A_n, B_n, C_n\subset \Omega_n^\ast$ the following events. 

Let $A_n\subset \Omega_n^\ast$ denote the set of all simplicial complexes $Y$ such that for any two vertices $v, w\in Y$ the intersection
$\lk_y(v)\cap \lk_Y(w)$ is connected. 

$B_n\subset \Omega_n^\ast$ will denote the set of all simplicial complexes $Y$ which have no edges $e\subset Y$ of degree zero, i.e. every edge $e\subset Y$ is incident to a 2-simplex $\sigma\subset Y$. 

And finally, the symbol $C_n\subset \Omega_n^\ast$ will denote the set of all simplicial complexes $Y$ such that every triple of its vertexes has a common neighbour. 

We note that any $Y\in A_n\cap B_n \cap C_n$ is simply connected. Indeed, taking the cover by the closed stars of vertices we see that 
the intersection $\St(v)\cap \St(w) $ is connected; if $(vw)\not\subset Y$ then it follows from the definition of $A_n$ and 
if $(vw)\subset Y$ then $\St(vw)$ is contractible (and hence connected) and has nontrivial intersection with $\lk_Y(v)\cap\lk_Y(w)$ as
follows from our assumption $Y\in B_n$; this shows that 
$\St(v)\cap \St(w)$ is connected. Finally we apply the Nerve Lemma \ref{lemnerve} using our assumption $Y\in C_n$. 

To complete the proof we only need to show that $\up(A_n)\to 1$ and $\up(B_n)\to 1$; Lemma \ref{common} tells us that $\up(C_n)\to 1$. 

Consider two fixed vertexes $v, w\in Y$ and consider the intersection $\lk_Y(v)\cap \lk_Y(w)$. By Lemma \ref{linkofkvert} this intersection 
is a random simplicial complex with respect to the lower measure with probability parameters $p'_\tau = p_\tau p_v p_w$, i.e. it is also a lower model random simplicial in the medial regime. By Corollary \ref{corconn} the intersection $\lk_Y(v)\cap \lk_Y(w)$ is disconnected with probability at most $C\exp(-\frac{n^{1/2}}{2})$ and hence the expected number of pairs of vertices with disconnected 
$\lk_Y(v)\cap \lk_Y(w)$ is bounded above by
$$C n^2 \exp(-\frac{n^{1/2}}{2})\to 0.$$
This proves that $\up(A_n)\to 1$. 

The proof of $\up(B_n)\to 1$ is similar. By Theorem 6.2 from \cite{farber}, the link of an edge $e=(vw)\subset Y$ is a random simplicial complex with respect to the lower model with probability parameters $$p'_\tau= p_\tau p_v p_u\ge p^3$$ and hence the probability that an edge $e$ has empty link is bounded above by 
$$(1-p^3)^{n-1} \le \exp(-p^3 n)$$
for $n$ large enough. Thus, the expected number of edges $e\subset Y$ with empty links is at most 
$$\binom {n+1} 2 \cdot e^{-p^3 n} \to 0,$$
implying $\up(B_n)\to 1$ by the first moment method. This completes the proof of Proposition \ref{propsimplyconn}. 
\end{proof}

\section{Vanishing of the Betti numbers}
The main result of this section states that homogeneous (see Definition \ref{homog}) lower model random simplicial complexes in the medial regime (see (\ref{medial})) have trivial rational homology in every dimension not exceeding 
$$ \log_2\ln n -\log_2 a - 1 -\delta_0,$$
where $p = e^{-a}$ as in (\ref{medial1}) and $\delta_0>0$ is any constant. 

\begin{theorem}\label{thmtrivialhomology}
Let $Y\in \Omega_n^\ast$ be a homogeneous random simplicial complex with respect to the lower probability measure in the medial regime. 
Then for any constant $\delta_0>0$, the rational homology of $Y$ vanishes, 
$$H_j(Y;\Bbb Q) = 0,$$
for all $$0 < j\,  \le\,    \log_2\log_{(p^{-1})} n -1-\delta_0,$$
a.a.s. 
\end{theorem}

The proof of Theorem \ref{thmtrivialhomology}  given below uses Garland's method as described in  \cite{ballmann}. 

Given a graph $G$ we denote by $\mathcal{L} = \mathcal{L}(G)$ the normalised Laplacian of $G$. All eigenvalues of $\mathcal{L}$ lie in $[0,2]$ and the multiplicity of the eigenvalue $0$ equals the number of connected components of $G$. 
Let $\kappa(G) > 0$ denote the smallest non-zero eigenvalue of $\mathcal{L}$; the quantity $\kappa(G)$ is known as 
 the \textit{spectral gap} of $G$. 

Given a simplicial complex $X$ and a simplex $\sigma\in X$, 
let $L_\sigma$ denote the $1$-skeleton of the link $\lk_X(\sigma)$ and let $\kappa_\sigma = \kappa(L_\sigma)$ denote the spectral gap of the graph $L_\sigma$.

The following result is well-known, see \cite{ballmann}:
\begin{theorem}\label{thmballmann}
Let $\ell\geq 0$ be a non-negative integer. If $X$ is a finite $(\ell+2)$-dimensional simplicial complex such that for every $\ell$-dimensional simplex $\sigma\in X$ the link $L_\sigma$ 
 is a non-empty connected graph with spectral gap satisfying $$\kappa\left(L_\sigma\right)>1-\frac {1}{\ell+2},$$ then $$H^{\ell +1}(X;\mathbb Q)=0.$$
\end{theorem}

Recall that by Corollary \ref{corconn} and Proposition \ref{propsimplyconn} the lower random complex $Y$ in the medial regime is connected and simply connected. Thus, Theorem \ref{thmtrivialhomology}  follows once we have established:

\begin{lemma}\label{lemspectral}
Let $Y\in \Omega_n^\ast$ be a homogeneous random simplicial complex with respect to the lower probability measure in the medial regime, see  (\ref{medial}). 
Then $Y$ has the following property with probability tending to $1$ as $n\to \infty$: for every $\ell$-dimensional simplex $\sigma\subset Y$, where 
\begin{eqnarray}\label{ell}
%0\leq\ell<\left\lfloor \log_2\ln n -\log_2(a)\right\rfloor-1 ,\\
0\leq\ell \, \le \,  \log_2\log_{(p^{-1})} n -2-\delta_0
\end{eqnarray}
 the link $L_\sigma$ is non-empty, connected and its spectral gap satisfies
 $\kappa_\sigma>1-\frac{1}{\ell +2}.$
\end{lemma}
%Theorem~\ref{thmtrivialhomology} will then follow from Lemma~\ref{lemspectral} and application of 
%
%to $X = Y^{(\ell+2)}$ for every $\ell<\left\lfloor k(n,p)-1\right\rfloor$.

\begin{proof} Fix a simplex $\sigma\subset \Delta_n$ of dimension $\ell$ and let $\Delta'_\sigma\subset \Delta_n$ denote the simplex spanned by those vertexes of $[n]$ which are not in $\sigma$; clearly $\dim \Delta'_\sigma= n-\ell-1$. 
Consider a random simplicial complex  $Y\in \Omega_n^\ast$ containing $\sigma$. The 1-skeleton $L_\sigma$ of the link $\lk_Y(\sigma)$ 
 is a random subgraph of $\Delta'_\sigma$ and according to Theorem 6.2 from \cite{farber} the graph $L_\sigma$ is a random graph with respect to the lower probability measure 
with vertex and edge probability parameters given by the formulae
\begin{eqnarray}\label{eqnprobpara}
p'_v = p_v\cdot \prod_{\tau\subseteq \sigma} p_{v\tau}\quad\text{and}\quad p'_{e} = p_{e}\cdot \prod_{\tau\subseteq \sigma} p_{e\tau}.
\end{eqnarray}

Since  $p\leq p_\tau\leq P$ for every simplex $\tau$ we obtain the following bounds on the probability parameters $p'_v$ and $p'_e$ of the graph $L_\sigma$
\begin{align}\label{eqnprobboundslink}
p^{2^{\ell + 1}} \leq p'_v\leq P^{2^{\ell+1}} \quad \mbox{and} \quad p^{2^{\ell + 1}} \leq p'_e\leq P^{2^{\ell+1}}.
\end{align}

Since by assumption $Y$ is homogeneous it follows that the link $L_\sigma$ is homogeneous as well, i.e. $p'_v=p'_{v'}$ and $p'_e=p'_{e'}$
for any vertices $v, v'$ and edges $e, e'$ of $L_\sigma$.  

%Note that due to our assumption that $p_\sigma$ depends only on the dimension of $\sigma$, we obtain that $p'_v=p'_{v'}$ and 
%$p'_e=p'_{e'}$ for any two vertices $v, v'$ and any two edges $e, e'$ of $\Delta'_\sigma$. 

The function $f_0^\sigma$ counting the number of vertices of $L_\sigma$, 
$Y\mapsto f_0(L_\sigma)$, is a random variable and its expectation $\E(f_0^\sigma)$ satisfies
$$(n-\ell) p^{2^{\ell+1}}\le \E(f_0^\sigma) \le (n-\ell) P^{2^{\ell+1}}.$$

From now on we shall assume (because of (\ref{ell})) that 
$$\ell+1 \le  \log_2\log_{(p^{-1})}n -1 - \delta_0$$
where $\delta_0>0$ is a constant. We can write 
$$\ell+1 =  \log_2\log_{(p^{-1})}n - x$$
where $x=x(\ell)\, \ge 1+\delta_0$. Then 
$$p^{2^{\ell+1}} = n^{-2^{-x}}\quad \mbox{and}\quad P^{2^{\ell+1}} = n^{-\lambda 2^{-x}}$$
where $\lambda=\log_p P$ is a constant,  $0<\lambda<1.$ Thus we see that
\begin{eqnarray}\label{greater}\E(f_0^\sigma) \ge (n-\ell)n^{-2^{-x}}\ge \frac{1}{2} n^{1-2^{-x}}
\end{eqnarray}
and similarly, 
\begin{eqnarray}\label{smaller}
\E(f_0^\sigma) \le n^{1-\lambda 2^{-x}}.\end{eqnarray}

Since $f_0^\sigma$ is a binomial random variable we may apply Chernoff's inequality (see Corollary 2.3 in \cite{JLR}) which states that for any $0<\epsilon<3/2$ the probability that $f_0^\sigma$ deviates from its expectation $\E(f_0^\sigma)$ by more than 
$\epsilon \E(f_0^\sigma)$ is at most $ 2 \exp(-\frac{\epsilon^2}{3}\E(f_0^\sigma ))$. 
Thus, probability that $f_0^\sigma$ is smaller than $\frac{1}{4}n^{1-2^{-x}}$ is bounded above by 
$$2\cdot \exp\left(-\frac{n^{1-2^{-x}}}{24} \right)\le 
2\cdot \exp\left(-{n^{1/2}}\right).$$ 
Similarly, the probability that $f_0^\sigma$ is larger than $2n^{1-\lambda 2^{-x}}$ is smaller than $2\cdot \exp(-{n^{1/2}})$. Hence we see that 
the probability that for some $\ell$ satisfying
$$\ell+1\le \log_2\log_{(p^{-1})}n - 1 -\delta_0$$ the inequality 
\begin{eqnarray}\label{range}
\frac{1}{4} n^{1-2^{-x}}\, \le\,  f_0^\sigma \, \le\,  2n^{1-\lambda 2^{-x}}
\end{eqnarray}
is violated is smaller than
$$4e^{-n^{1/2}}\cdot (n+1)^{\log_2\log_{(p^{-1})}n},$$
it is easy to see that this quantity tends to zero as $n\to \infty$. 
Thus, asymptotically almost surely, the graph $L_\sigma$ is an Erd\H{o}s-R\'{e}nyi  random graph on the number of vertices $N=f_0^\sigma$ satisfying (\ref{range}). The edge probability $\rho$ of $L_\sigma$ satisfies the inequalities
$$p^{2^{\ell+1}}\le \rho\le P^{2^{\ell+1}}.$$

We shall use the following result about the spectral gap of the Erd\H{o}s-R\'{e}nyi random graphs which is a corollary of Theorem 1.1 from \cite{hoffman}. Consider a random Erd\H{o}s-R\'{e}nyi graph $G\in G(N, \rho)$ such that 
\begin{eqnarray}\label{rho}
\rho\, \ge\,  \frac{(1+\delta)\log N}{N},
\end{eqnarray}
for some fixed $\delta>0$. Then for any $c\ge 1$ there exists an integer $N_{c, \delta}$ such that for any $N>N_{c, \delta}$ the graph $G$ is connected and 
\begin{eqnarray}\label{kap}
\kappa(G) >1-\frac{1}{c}\end{eqnarray}
with probability at least $1-N^{-\delta}$. 

We shall apply this statement with $c=\ell+2$ and $\delta=3\ell$. 
Using (\ref{range}) %(\ref{greater}) and (\ref{smaller}) 
we obtain
$$
\frac{\rho N}{\log N} \ge 
\frac{p^{2^{\ell+1}} f_0^\sigma}{\log f_0^\sigma} \ge \frac{1}{2} \frac{n^{-2^{-x} } n^{1-2^{-x}}} {(1-\lambda 2^{-x} ) \log n +1}\ge
\frac{1}{4} \frac{n^{1-2^{1-x} }} { \log n}
\ge \frac{1}{4} \frac{n^{1-2^{-\delta_0} }} { \log n}\quad \ge 1+\delta=3\ell+1.
$$
Hence we see that for any $n\ge M_0$ (where $M_0$ is an integer depending only on the value of $\delta_0$) the inequality (\ref{rho}) will be violated for a given simplex $\sigma$ with probability at most 
$$
n \cdot N^{-3\ell} \le n \cdot \left(\frac{1}{4} n^{1-2^{-x}} \right)^{-3\ell}\le n^{-\frac{3\ell}{2}},
$$
provided $N\ge \frac{1}{4} n^{1-2^{-x}}$. 
Here the factor $n$ 
takes into account the fact that we are applying inequality (\ref{rho}) a number of times, for each possible value of $N$, and the range of values of $N$ is bounded above by $2n^{1-\lambda 2^{-x}}\le n$
according to (\ref{range}).

Therefore the expected number of simplexes $\sigma$ with $\dim \sigma \, \le\,  \log_2\log_{(p^{-1})}n-2 -\delta_0$, for which (\ref{rho}) is violated is bounded above by
$$n^{\ell+1} \cdot \left(4e^{-n^{1/2}}\cdot (n+1)^{\log_2\log_{(p^{-1})}n}+ n^{- \frac{3\ell}{2} } \right)$$

 and this quantity obviously tends to zero. Thus, with probability tending to 1, the spectral gap inequality (\ref{kap}) will be satisfied for all simplexes $\sigma$ in the indicated range of dimensions. This completes the  proof of Lemma \ref{lemspectral}. 
\end{proof}

\section{Proof of Theorem \ref{thm2}}\label{prfthm2}

The probability that no $(n-1)$-dimensional simplexes is included into $Y$ is 
$$\prod_{\dim \sigma = n-1}(1-p_\sigma) \le (1-p)^{n+1}$$
which converges to $0$ since $0<p<1$ is a constant. This proves statement (1). 

The proofs of statements (2) and (3) are based on Theorem \ref{thm1} and the duality relation given by Theorem \ref{dual}. Indeed, let $Y$ be a random simplicial complex with respect to the upper model in the medial regime, i.e. we assume that the probability parameters $p_\sigma$ satisfy
$$0<e^{-a} =p \le p_\sigma\le P=e^{-A}<1.$$
Consider the dual system of probability parameters $p'_\sigma =1- p_{\check\sigma}$ (see (\ref{dualsys})) which satisfies 
$$0<e^{-A'}=1-P \le p'_\sigma\le 1-p=e^{-a'}<1, $$
where $a'$ and $A'$ are defined in (\ref{prime}).
Next, we use the isomorphism $\cc$ of Theorem \ref{dual} and the duality for the Betti numbers (\ref{betti}). The complex $\cc(Y)$ is a random simplicial complex in the lower model with respect to the system of probability parameters $p'_\sigma$. Hence by Theorem \ref{thm1}, the dimension of the complex $\cc(Y)$ satisfies 
\begin{eqnarray}\label{dimc}
\lfloor \beta(n,A')\rfloor -1 \le \dim \cc(Y) \le \beta(n, a')-1+\epsilon_0,
\end{eqnarray}
a.a.s. where $\epsilon_0>0$ is an arbitrary constant. Since the maximal dimension $d$ such that $\cc(Y)$ contains the skeleton 
$\Delta_n^{(d)}$ equals $n-2-\dim Y$, the inequality (\ref{dimc}) implies statement (2) of Theorem \ref{thm2}. 

To prove the third statement we observe that the reduced Betti numbers of $\cc(Y)$ vanish in all dimensions except possibly
$$\log_2\ln n  -\log_2A' -1 -\delta_0 <j \le   \log_2\ln n +   \log_2\log_2\ln n -\log_2a' -1 +\epsilon_0,$$
Since $b_j(Y) = b_{n-2-j}(\cc(Y))$ (cf. (\ref{betti})), we obtain that the Betti numbers $b_j(Y)$ vanish except possibly for 
$$\log_2\ln n  -\log_2A' +1 -\delta_0 \, <\, n-j\, \le   \log_2\ln n +   \log_2\log_2\ln n -\log_2a' +1 +\epsilon_0.$$
This completes the proof.


\begin{thebibliography}{99} 

\bibitem{ballmann} W. Ballmann and J. \'{S}wi\c{a}tkowski, 
\textit{On $L^2$ cohomology and property (T) for automorphism groups of polyhedral cell complexes},
GAFA, {\bf{7}} (1997) 615 -- 645. 

\bibitem{bjornertop} A. Bj\"orner, \textit{Topological methods}. Handbook of combinatorics, Vol. 1, 2, pp. 1819–1872, Elsevier Sci. B. V., Amsterdam, 1995. 


\bibitem{bjorner} A. Bj\"orner and M. Tancer, \textit{Combinatorial Alexander duality - a short and elementary proof}, Discrete and Computational Geometry,
\textbf{42}(2009), 586-593.

\bibitem{Bo} B. Bollobas, Random graphs, Cambridge University Press, 2001.

\bibitem{cooley} O. Cooley, M. Kang, Y. Person, \textit{
Largest components in random hypergraphs.}
Combin. Probab. Comput. {\bf 27} (2018), no. 5, 741–762. 

\bibitem{costa} A. Costa and M. Farber, \textit{Random Simplicial Complexes}, Configuration spaces, 129–153, Springer INdAM Ser., 14, Springer, 2016.

\bibitem{costa1} A. Costa and M. Farber, \textit{Large Simplicial Random Complexes I}, J. Topology and Anal. 8 (2016), no. 3, 399–429.

\bibitem{costa2} A. Costa and M. Farber, \textit{Large random simplicial complexes, II; the fundamental group}. J. Topology and Anal. 9 (2017), no. 3, 441–483.

\bibitem{costa3} A. Costa and M. Farber, \textit{Large Random Simplicial Complexes III: The Critical Dimension}. J. Knot Theory Ramifications 26 (2017), no. 2, 1740010

\bibitem{ER} P.\ Erd\H{o}s and A.\ R\'enyi, {\it On the evolution of 
random graphs}, Publ.\ Math.\ Inst.\ Hungar.\ Acad.\ Sci.\ {\bf 5}
(1960), 17--61.

\bibitem{farber} M. Farber, L. Mead and T. Nowik, \textit{Random Simplicial Complexes, Duality and The Critical Dimension}, arXiv:1901.09578, %to appear in Journal of Topology and Analysis (2019)

\bibitem{hatcher} A. Hatcher, \textit{Algebraic Topology}, Cambridge University Press (2002)

\bibitem{hoffman} C. Hoffman, M. Kahle, E. Paquette, \textit{Spectral gaps of random graphs and applications to random topology}, arXiv:1201.0425.




%\bibitem{Ksurvey} M. Kahle, \textit{Topology of random simplicial complexes: a survey},  
%Algebraic topology: applications and new directions, 201–221, Contemp. Math., 620, Amer. Math. Soc., Providence, RI, 2014.

\bibitem{JLR} S. Janson, T. Luczak, A.Rucinski, \textit{Random graphs}, 2000. 


\bibitem{kahle} M. Kahle, \textit{Topology of random clique complexes}, Discrete Math. 309 (2009), no. 6,
1658 -- 1671. 

\bibitem{kahle1} M. Kahle, \textit{Topology of random simplicial complexes: a survey}. In: 
Algebraic topology: applications and new directions, 201–221, Contemp. Math., 620, Amer. Math. Soc., Providence, RI, 2014.

%\bibitem{Kahle3} M. Kahle, \textit{Sharp vanishing thresholds for cohomology of random flag complexes}, 
%Ann. of Math. (2), \textbf{179} (2014), no. 3, pp. 1085 - 1107. %1085 -- 1107.


\bibitem{linial} N. Linial and R. Meshulam,  \textit{Homological connectivity of random 2-complexes}, Combinatorica \textbf{26} (2006), 475-487.


\bibitem{meshulam} R. Meshulam and N. Wallach, \textit{Homological connectivity of random k-complexes}, Random Structures \& Algorithms \textbf{34} (2009), 408-417.

\end{thebibliography}
\end{document}